\documentclass[11pt, a4paper]{article}

\usepackage[utf8]{inputenc}
\usepackage{amsmath, amsfonts, amssymb, amsthm, cases, csquotes, enumitem, hyperref, mathrsfs, mathtools, tikz, tikz-cd}
\usepackage[style=alphabetic]{biblatex}

\theoremstyle{definition}

\newtheorem{definition}{Definition}[section]
\newtheorem{theorem}{Theorem}[section]
\newtheorem{proposition}{Proposition}[section]
\newtheorem{corollary}{Corollary}[section]
\newtheorem{lemma}{Lemma}[section]
\newtheorem{example}{Example}[section]

\usepackage[left=2cm,right=2cm,top=2cm,bottom=2cm,bindingoffset=0cm]{geometry}

\title{Homological dimensions of analytic Ore extensions}
\author{Petr Kosenko}

\newcommand{\Addresses}{{
		\bigskip
		\footnotesize
		
		P.~Kosenko, \textsc{Department of Mathematics, Higher School of Economics, Moscow, Usacheva str. 6}\par\nopagebreak
		\textit{E-mail address}, P.~Kosenko: \texttt{prkosenko@edu.hse.ru}
}}

\begin{document}
	\maketitle
\begin{abstract}
	If $A$ is an algebra with finite right global dimension, then for any automorphism $\alpha$ and $\alpha{\text{-derivation }} \delta$ the right global dimension of $A[t; \alpha, \delta]$ satisfies
	\[
	\text{rgld} \, A \le \text{rgld} \, A[t; \alpha, \delta] \le \text{rgld} \, A + 1.
	\]
	We extend this result to the case of holomorphic Ore extensions and smooth crossed products by $\mathbb{Z}$ of $\hat{\otimes}$-algebras.
\end{abstract}

\section{Introduction}
\label{introduction}
We will start this paper by recalling the following well-known theorem:

\begin{theorem}[\cite{book:71679}, Theorem 4.3.7]
	\label{classic}
	If $R$ is a ring then the following estimate takes place for every $n \in \mathbb{N}$:
	\[
	\text{rgld} \, R[x_1, \dots x_n] = n + \text{rgld} \, R.
	\]
\end{theorem}

The importance of this theorem lies in the fact that it immediately yields the Hilbert's syzygy theorem in the case when $R$ is a field (see [\cite{book:71679}, Corollary 4.3.8]).

This fact can be, indeed, generalized to Ore extensions $R[t; \alpha, \delta]$, as shown in \cite{NonComm2001}. It turns out that if the global dimension of $R$ is finite, then the global dimension of $R[t; \alpha, \delta]$ either stays the same, or increases by one.

In this paper we adapt the arguments used in \cite[ch. 7.5]{NonComm2001} to the topological setting in order to obtain the estimates for the right homological dimensions of holomorphic(topological) Ore extensions (see \cite[ch. 4.1]{PirkAM}) and smooth crossed products by $\mathbb{Z}$ (see \cite{schweitzer1993dense} and \cite{phillips1994representable}). 


Below we state the result in the purely algebraic situation, which is provided in \cite{NonComm2001} and then we present its topological version.\\
\textbf{Remark.} There is an ambiguity in defining Ore extensions, which will be demonstrated below, so, to state the result in the algebraic setting, we need to fix an appropriate definition of Ore extensions:

\begin{definition}
\label{ore1}
Let $A$ be an algebra, $\alpha \in \text{End}(A)$ and let $\delta : A \rightarrow A$ be a $\mathbb{C}$-linear map, such that the following relation holds for every $a, b \in A$:
	$$
	\delta(ab) = \delta(a)b + \alpha(a)\delta(b).
	$$
Let us call such maps $\alpha$-derivations.

Then the \textit{Ore extension} of $A$ w.r.t $\alpha$ and $\delta$ is the vector space 
\[
A[t; \alpha, \delta] =\left\lbrace  \sum_{i=0}^n a_i t^i : a_i \in A \right\rbrace 
\]
with the multiplication defined uniquely by the following conditions:
\begin{enumerate}[label=(\arabic*)]
	\item The relation $ta = \alpha(a) t + \delta(a)$ holds for any $a \in A$.
	\item The natural inclusions $A \hookrightarrow A[t; \alpha, \delta]$ and $\mathbb{C}[t] \hookrightarrow A[t; \alpha, \delta]$ are algebra homomorphisms.
\end{enumerate}
Also, if $\delta = 0$ and $\alpha$ is invertible, then one can define the \textit{Laurent Ore extension} of $A$ 
$$
A[t, t^{-1}; \alpha] = \left\lbrace \sum_{i = -n}^n a_i t^i : a_i \in A \right\rbrace 
$$
 
with the multiplication defined the same way.
\end{definition}

The thing is, the authors of \cite{NonComm2001} denote slightly different type of algebras by $A[t; \alpha, \delta]$:
\begin{definition}[\cite{NonComm2001}, pp. 1.2.1-1.2.6]
	\label{ore2}
Let $A$ be an algebra, $\tilde{\alpha} \in \text{End}(A)$ and let $\tilde{\delta} : A \rightarrow A$ be a $\mathbb{C}$-linear map, such that the following relation holds for every $a, b \in A$:
	$$
	\tilde{\delta}(ab) = \tilde{\delta}(a) \tilde{\alpha}(b) + a \tilde{\delta}(b).
	$$
Let us call such maps opposite $\alpha$-derivations.
Then the \textit{(opposite) Ore extension} of $A$ w.r.t $\tilde{\alpha}$ and $\tilde{\delta}$ is the vector space 
\[
A_{\text{op}}[t; \tilde{\alpha}, \tilde{\delta}] = \left\lbrace  \sum_{i=0}^n t^i a_i : a_i \in A \right\rbrace 
\] 
with the multiplication defined uniquely by the following conditions:
\begin{enumerate}[label=(\arabic*)]
	\item The relation $at = t \tilde{\alpha}(a) + \tilde{\delta}(a)$ holds for any $a \in A$.
	\item The natural inclusions $A \hookrightarrow A_{\text{op}}[t; \tilde{\alpha}, \tilde{\delta}]$ and $\mathbb{C}[t] \hookrightarrow A_{\text{op}}[t; \tilde{\alpha}, \tilde{\delta}]$ are algebra homomorphisms.
\end{enumerate}
Also, if $\tilde{\delta} = 0$ and $\tilde{\alpha}$ is invertible, then one can define the \textit{Laurent Ore extension} of $A$
\[
A_{\text{op}}[t, t^{-1}; \tilde{\alpha}] =\left\lbrace  \sum_{i = -n}^n t^i a_i : a_i \in A \right\rbrace 
\]
with the multiplication defined the same way.
\end{definition}

It is easily seen that in the case of invertible $\alpha$, the following algebra isomorphisms take place:
\begin{equation}
	\label{Ore isomorphisms}
	A[x; \alpha, \delta] \cong A_{\text{op}}[x, \alpha^{-1}, -\delta \alpha^{-1}], \quad A[x, x^{-1}; \alpha ] \cong A_{\text{op}}[x, x^{-1}; \alpha^{-1}].
\end{equation}

Throughout the paper, we will work with Ore extensions in the sense of Definition \ref{ore1} (if not stated otherwise). 

Now we are ready to state the result in the purely algebraic case, which is, essentially, contained in the \cite[Theorem 5.7.3]{NonComm2001}:

\begin{theorem}[\cite{NonComm2001}, Theorem 5.7.3]
	Let $A$ be an algebra, let $\sigma$ be an automorphism and let $\delta$ be a $\sigma$-derivation, in the sense of a Definition \ref{ore2}. Denote the right global dimension of a ring $R$ by $\text{dgr}(R)$. Then the following estimates hold:
	\begin{enumerate}[label=(\arabic*)]
		\item $\text{dgr} \, A \le \text{dgr} \, A_{\text{op}}[t; \sigma, \delta] \le \text{dgr} \, A + 1 \text{ if } \text{dgr} \, R < \infty$
		\item $\text{dgr} \, A \le\text{dgr} \, A_{\text{op}}[t, t^{-1}; \sigma] \le \text{dgr} \,A + 1$
		\item $\text{dgr} \, A_{\text{op}}[t, \sigma] = \text{dgr} \, A + 1$
		\item $\text{dgr} \, A[t, t^{-1}] = \text{dgr} \, A + 1$
	\end{enumerate}
\end{theorem}
\noindent
\textbf{Remark.} In fact, the above theorem still holds if we replace $A_{\text{op}}[t; \sigma, \delta]$ with $A[t; \sigma, \delta]$ due to \eqref{Ore isomorphisms}.




This paper is organized as follows: in the Section \ref{homological dimensions} we recall the important notions related to homological properties of topological modules, in particular, we provide definitions of homological dimensions for topological algebras and modules. In the Section \ref{Estimates for the bidimension and projective global dimensions of holomorphic Ore extensions} we compute the estimates for the homological dimensions of holomorphic Ore extensions; we use the bimodules of relative differentials to construct the required projective resolutions. In the Section \ref{Homological dimensions of smooth crossed products} we compute the estimates for the smooth crossed products by $\mathbb{Z}$.

In the Appendix \ref{appendix A} we provide the computations of algebraic and topological bimodules of relative differentials for different types of Ore extensions.

\section*{Acknowledgments}
I want to thank Alexei Pirkovskii for his assistance in writing this paper, Alexei's comments and ideas were very helpful and greatly improved this paper.
\section{Homological dimensions}
\label{homological dimensions}
\textbf{Remark.} All algebras in this paper are considered to be complex, unital and associative. Also, we will be working only with unital modules.

\subsection{Notation}
Let us introduce some notation (see \cite{Helem1986} and \cite{PirkVdB} for more details).

For locally convex Hausdorff spaces $E, F$ we denote the completed projective tensor product of locally convex spaces $E, F$ by $E \hat{\otimes} F$.

\begin{definition}
	A \textit{Fr\'echet} space is a complete metrizable locally convex space.
\end{definition}
In other words, a locally convex space $X$ is a Fr\'echet space if and only if the topology on $X$ can be generated by a countable family of seminorms.

Denote by \textbf{LCS}, \textbf{Fr} the categories of complete locally convex spaces and Fr\'echet spaces, respectively. Also we will denote the category of vector spaces by $\textbf{Lin}$.

For a detailed introduction to the theory of locally convex spaces and algebras, and relevant examples, the reader can see \cite{book:29603}, \cite{book:955359}, \cite{book:338950}, or \cite{Helem2006}.

%
\begin{definition}
	Let $A$ be a locally convex space with a multiplication $\mu : A \times A \rightarrow A$, such that $(A, \mu)$ is an algebra.
	\begin{enumerate}[label=(\arabic*)]
		\item If $\mu$ is separately continuous, then $A$ is called \textit{a locally convex algebra}. 
		\item If $A$ is a complete locally convex space, and $\mu$ is jointly continuous, then $A$ is called a $\hat{\otimes}$-\textit{algebra}.
	\end{enumerate}
	If $A, B$ are $\hat{\otimes}$-algebras and $\eta : A \rightarrow B$ is a continuous unital algebra homomorphism, then the pair $(B, \eta)$ is called a $A$-$\hat{\otimes}$\textit{-algebra}.
	
	A $\hat{\otimes}$-algebra with the underlying locally compact space which is a Fr\'echet space is called a \textit{Fr\'echet algebra}.
\end{definition}

\begin{definition}
	A locally convex algebra $A$ is called $m$-convex if the topology on it can be defined by a family of submultiplicative seminorms.
\end{definition}

\begin{definition}
	A complete locally $m$-convex algebra is called an \textit{Arens-Michael algebra}.
\end{definition}

\begin{definition}
	Let $A$ be a $\hat{\otimes}$-algebra and let $M$ be a complete locally convex space with a structure of a topological $A$-module w.r.t. the locally convex topology on $M$. Also suppose that the natural map $A \times M \rightarrow M$ is jointly continuous. Then we will call $M$ a left $A$-$\hat{\otimes}$\textit{-module}. In a similar fashion we define right $A$-$\hat{\otimes}$\textit{-modules} and $A$-$B$-$\hat{\otimes}$\textit{-bimodules}.
	
	A $\hat{\otimes}$-module over a Fr\'echet algebra which is itself a Fr\'echet space is called a Fr\'echet $A$-$\hat{\otimes}$-module.
\end{definition}

For $\hat{\otimes}$-algebras $A, B$ we denote
\[
\begin{aligned}
A\text{-}\textbf{mod} &= \text{the category of (unital) left } A\text{-}\hat{\otimes}\text{-modules}, \\
\textbf{mod}\text{-}A &= \text{the category of (unital) right } A\text{-}\hat{\otimes}\text{-modules}, \\
A\text{-}\textbf{mod}\text{-}B &=  \text{the category of (unital) } A\text{-}B\text{-}\hat{\otimes}\text{-bimodules}. \\
\end{aligned}
\]

More generally, for a fixed category $\mathcal{C} \subseteq \textbf{LCS}$ we denote
\[
\begin{aligned}
A\text{-}\textbf{mod}(\mathcal{C}) &= \text{the category of left (unital) } A\text{-}\hat{\otimes}\text{-modules whose underyling LCS belong to } \mathcal{C}, \\
\textbf{mod}\text{-}A(\mathcal{C})&= \text{the category of right (unital) } A\text{-}\hat{\otimes}\text{-modules whose underyling LCS belong to } \mathcal{C}, \\
A\text{-}\textbf{mod}\text{-}B(\mathcal{C}) &=  \text{the category of (unital) } A\text{-}B\text{-}\hat{\otimes}\text{-bimodules whose underyling LCS belong to } \mathcal{C}. \\
\end{aligned}
\]
Complexes of $A$-$\hat{\otimes}$-modules for a $\hat{\otimes}$-algebra $A$
$$
\dots \xrightarrow{d_{n+1}} M_{n+1} \xrightarrow{d_n} M_n \xrightarrow{d_{n-1}} M_{n-1} \xrightarrow{d_{n-2}} \dots
$$
are usually denoted by $\{M, d\}$.

\begin{definition}
	Let $A$ be a $\hat{\otimes}$-algebra and consider a left $A$-$\hat{\otimes}$-module $Y$ and a right $A$-$\hat{\otimes}$-module $X$.
	\begin{enumerate}[label=(\arabic*)]
		\item A bilinear map $f : X \times Y \longrightarrow Z$, where $Z \in \textbf{LCS}$, is called $A$-balanced if $f(x \circ a, y) = f(x, a \circ y)$ for every $x \in X, y \in Y, a \in A$.
		\item A pair $(X \hat{\otimes}_A Y, i)$, where $X \hat{\otimes}_A Y \in \textbf{LCS}$, and $i : X \hat{\otimes}_A Y \longrightarrow X \hat{\otimes}_A Y$ is a continuous $A$-balanced map, is called the \textit{completed projective tensor product of} $X$ and $Y$, if for every $Z \in \textbf{LCS}$ and  continuous $A$-balanced map $f : X \times Y \longrightarrow Z$ there exists a unique continuous linear map $\tilde{f} : X \hat{\otimes}_A Y \longrightarrow Z$ such that $f = \tilde{f} \circ i$.
	\end{enumerate}
\end{definition}

For the proof of the existence and uniqueness of complete projective tensor products of $\hat{\otimes}$-modules the reader can see \cite[ch. 2.3-2.4]{Helem1986}. In this paper we would like to keep in mind a trivial, but nonetheless useful example:
\begin{example}
	Let $A$ be a $\hat{\otimes}$-algebra, and consider a left $A$-$\hat{\otimes}$-module $M$. Then
	\begin{equation}
	\label{a simple tensor product}
		A \hat{\otimes}_A M \xrightarrow{\sim} M, \quad a \otimes m \mapsto a \cdot m,
	\end{equation}
	is a topological isomorphism of left $A$-$\hat{\otimes}$-modules. Similar isomorphisms can be constructed for right $A$-$\hat{\otimes}$-modules.
\end{example}
\subsection{Projectivity and flatness}

The following definitions shall be given in the case of left modules; the definitions in the cases of right modules and bimodules are similar, just use the following category isomorphisms:
\[
\textbf{mod}\text{-}A \simeq A^{\text{op}}\text{-}\textbf{mod};~ A\text{-}\textbf{mod}\text{-}B \simeq (A \hat{\otimes} B^{\text{op}}) \text{-}\textbf{mod}.
\]

Let us fix a $\hat{\otimes}$-algebra $A$.

\begin{definition}
A complex of $A$-$\hat{\otimes}$-modules $\{M, d\}$ is called \textit{admissible} $\iff$ it splits in the category $\textbf{LCS}$. A morphism of $A$-$\hat{\otimes}$-modules $f : X \rightarrow Y$ is called \textit{admissible} if it is one of the morphisms in an admissible complex.
\end{definition}

\begin{definition}
	An additive functor $F : A\text{-}\textbf{mod} \rightarrow \textbf{Lin}$ is called  \textit{exact} $\iff$ for every admissible complex $\{M, d\}$ the corresponding complex $\{F(M), F(d)\}$ in $\textbf{Lin}$ is exact.
\end{definition}

\begin{definition}
	\indent
	\begin{enumerate}[label=(\arabic*)]
		\item A module $P \in A\text{-}\textbf{mod}$ is called \textit{projective} $\iff$ the functor $\text{Hom}_A(P, -)$ is exact.
		\item A module $Y \in A\text{-}\textbf{mod}$ is called \textit{flat} $\iff$ the functor $(-)\hat{\otimes}_A Y : \textbf{mod}\text{-}A \rightarrow \textbf{Lin}$ is exact.
		\item A module $X \in A\text{-}\textbf{mod}$ is called \textit{free} $\iff X$ is isomorphic to $A \hat{\otimes} E$ for some $E \in \mathcal{C}$.
	\end{enumerate}
\end{definition}
Due to the \cite[Theorem 3.1.27]{Helem1986}, any free module is projective, now we are going to prove that a free module is flat.
\begin{proposition}
	\label{free admissible}
	Let $X$ be a free left $A$-$\hat{\otimes}$-module. Then for every admissible sequence $\{M, d\}$ of right $A$-$\hat{\otimes}$-modules the complex $\{M \hat{\otimes}_A X, d \hat{\otimes} \text{Id}_X\}$, where $(M \hat{\otimes}_A X)_i = M_i \hat{\otimes}_A X$ splits in \textbf{LCS}.
\end{proposition}
\begin{proof}
	Fix an isomorphism $X \cong A \hat{\otimes} E$.
	
	If suffices to prove this proposition for short admissible sequences. However, notice that the sequence 
	\[
	\begin{tikzcd}
		0 \ar[r] & M_1 \hat{\otimes}_A X \ar[r, "d \otimes \text{Id}_X"] &  M_2 \hat{\otimes}_A X \ar[r, "d \otimes \text{Id}_X"] &  M_3 \hat{\otimes}_A X \ar[r] & 0  
	\end{tikzcd}
	\]
	is isomorphic to the sequence
	\[
	\begin{tikzcd}
	0 \ar[r] & M_1 \hat{\otimes} E \ar[r, "d \otimes \text{Id}_E"] &  M_2 \hat{\otimes} E \ar[r, "d \otimes \text{Id}_E"] &  M_3 \hat{\otimes} E \ar[r] & 0,  
	\end{tikzcd}
	\]
	and if $f : M_1 \rightarrow M_2$ is a coretraction in \textbf{LCS}, then $f \otimes \text{Id} : M_1 \hat{\otimes} E \rightarrow M_2 \hat{\otimes} E$ is a coretraction, as well, a similar argument holds for the retraction $M_2 \rightarrow M_3$.

\end{proof}

\begin{definition}
	Let $A$ be an algebra and let $M$ be a right $A$-module (or a $A$-bimodule, resp.). For any endomorphism $\alpha : A \rightarrow A$ denote by $M_\alpha$ a right $A$-module (or an $A$-bimodule, resp.), which coincides with $M$ as an abelian group (left $A$-module, resp.), and whose structure of right $A$-module is defined by $m \circ a = m \alpha(a)$. In a similar fashion one defines ${}_\alpha M$ for left modules.
\end{definition}

The proof of the following lemma is very similar to the proof of the \cite[Proposition 4.1.5]{Helem1986}.

\begin{lemma}
	\label{lemma1}
	Let $R$ be a $\hat{\otimes}$-algebra and let $A$ be a $R$-$\hat{\otimes}$-algebra.
	\begin{enumerate}[label=(\arabic*)]
		\item For every (projective/flat) left module $X \in R\text{-}\textbf{mod}$ the left $A$-$\hat{\otimes}$-module $A \hat{\otimes}_R X \in A\text{-}\textbf{mod}$ is (projective/flat). Moreover, for every projective (flat) right module $X \in \textbf{mod}\text{-}R$ the right $A$-module $X \hat{\otimes}_R A \in \textbf{mod}\text{-}A$ is (projective/flat).
		\item For every projective bimodule $P \in R\text{-}\textbf{mod}\text{-}R$ and $\alpha \in \text{Aut}(R)$ the bimodule \newline ${A_\alpha \hat{\otimes}_R P \hat{\otimes}_R A \in A\text{-}\textbf{mod}\text{-}A}$ is projective.
	\end{enumerate} 
\end{lemma}
\begin{proof}
	\indent
	\begin{enumerate}[label=(\arabic*)]
		\item The module $X \in R \text{-} \textbf{mod}$ is projective iff $P$ is a retract of a left free $R$-module (\cite[Theorem 3.1.27]{Helem1986}), in other words, there are $E \in \mathcal{C}$ and a retraction $\sigma : R \hat{\otimes} E \rightarrow X$. But then the map
		\[
		\text{Id}_A \otimes \sigma : A \hat{\otimes} E \cong A \hat{\otimes}_R R \hat{\otimes} E \rightarrow A \hat{\otimes}_R X
		\]
		is a retraction, as well. The statement of the lemma for flat modules is simpler, because for every $M \in A\text{-}\textbf{mod}$ we have the canonical isomorphism $M \hat{\otimes}_A A \hat{\otimes}_R X \cong M \hat{\otimes}_R X$.
		
		The proofs for the right modules are the same.
		\item First of all, notice that the following isomorphism of $A$-$R$-$\hat{\otimes}$-bimodules takes place: 
		\[
		A_\alpha \hat{\otimes}_R R \xrightarrow{\sim} A_\alpha.
		\] 
		In particular, $A_\alpha \cong A$ as a left $A$-$\hat{\otimes}$-module. Any projective bimodule is a retract of a free bimodule, in other words, there exist $E \in \mathcal{C}$ and a retraction $\sigma :   R \hat{\otimes} E \hat{\otimes} R \rightarrow P$. Notice that the map
		$$
		\text{Id}_{A_\alpha} \otimes \sigma \otimes \text{Id}_A : A \hat{\otimes} E \hat{\otimes} A \cong A_\alpha \hat{\otimes}_R R \hat{\otimes} E \hat{\otimes} R \hat{\otimes}_R A \rightarrow  A_\alpha \hat{\otimes}_R P \hat{\otimes}_R A
		$$
		is a retraction of $A$-$\hat{\otimes}$-bimodules, and $A \hat{\otimes} E \hat{\otimes} A$ is a free $A$-$\hat{\otimes}$-bimodule.
	\end{enumerate} 
\end{proof}

The following lemma proves that the \cite[Lemma 7.2.2]{NonComm2001} is true for topological modules and $\hat{\otimes}$-algebras.


\subsection{Homological dimensions}

\begin{definition}
	Let $X \in A\text{-}\textbf{mod}$. Suppose that $X$ can be included in a following admissible complex:
	$$
	0 \leftarrow X \xleftarrow{\varepsilon} P_0 \xleftarrow{d_0} P_1 \xleftarrow{d_1} \dots  \xleftarrow{d_{n-1}} P_n \leftarrow 0 \leftarrow 0 \leftarrow \dots,
	$$
	where every $P_i$ is a projective module. Then we will call such complex a \textit{projective resolution} of $X$ of \textit{length} $n$. By definition, the length of an unbounded resolution equals $\infty$. \textit{Flat resolutions} are defined similarly.
\end{definition}

This allows us to define the notion of a derived functor in the topological case, for example, see \cite[ch 3.3]{Helem1986}. In particular, $\text{Ext}_A^k(M, N)$ and $\text{Tor}^A_k(M, N)$ are defined similarly to the purely algebraic situation.

\begin{definition}
	Consider an arbitrary module $M \in A\text{-}\textbf{mod}(\mathcal{C})$ for some fixed category $\mathcal{C} \subseteq \textbf{LCS}$. Then due to \cite[Theorem 3.5.4]{Helem1986} following number is well-defined and we have the following chain of equalities:
	\[
		\begin{aligned}
			\text{dh}^{\mathcal{C}}_A(M) & := \min\{ n \in \mathbb{Z}_{\ge 0} : \text{Ext}^{n+1}_A (M, N) = 0 \text{ for every } N \in A\text{-}\textbf{mod}(\mathcal{C}) \} = \\ & = \{\text{the length of a shortest projective resolution of }M \text{ in } A\text{-}\textbf{mod}(\mathcal{C})  \} \in \{-\infty\} \cup [0, \infty].
		\end{aligned}
	\]
	As we can see, this number doesn't depend on the choice of the category $\mathcal{C}$:
	\[
	\begin{aligned}
	\text{dh}^{\mathcal{C}}_A(M) &= \min\{ n \in \mathbb{Z}_{\ge 0} : \text{Ext}^{n+1}_A (M, N) = 0 \text{ for every } N \in A\text{-}\textbf{mod}(\mathcal{C}) \} \le \\ & \le \min\{ n \in \mathbb{Z}_{\ge 0} : \text{Ext}^{n+1}_A (M, N) = 0 \text{ for every } N \in A\text{-}\textbf{mod} \} = \text{dh}^{\textbf{LCS}}_A(M),
	\end{aligned}
	\] 
	\[
	\begin{aligned}
	\text{dh}^{\textbf{LCS}}_A(M) & = \{\text{the length of a shortest projective resolution of }M \text{ in } A\text{-}\textbf{mod} \} \le \\ & \le \{\text{the length of a shortest projective resolution of }M \text{ in } A\text{-}\textbf{mod}(\mathcal{C}) \} = \text{dh}^{\mathcal{C}}_A(M).
	\end{aligned}
	\]
	So we will denote this invariant by $\text{dh}_A(M)$, and we will call it the \textit{projective homological dimension} of $M$. 
	
	If $A$ is a Fr\'echet algebra, and $M$ is a Fr\'echet module, then we can define the \textit{weak homological dimension} of $M$:
	$$
	\begin{aligned}
		\text{w.dh}_A(M)  &= \min \{ n \in \mathbb{Z}_{\ge 0} : \text{Tor}^A_{n+1}(N, M) = 0 \text{ and } \text{Tor}^A_n(N, M) \text{ is Hausdorff for every } N \in \textbf{mod}\text{-}A(\textbf{Fr})  \} = \\ & = \{\text{the length of the shortest flat resolution of } M \} \in \{-\infty\} \cup [0, \infty].
	\end{aligned}
	$$
\end{definition}


\begin{definition}
	Let $A$ be a $\hat{\otimes}$-algebra. Then we can define the following invariants of $A$:
	$$
	\text{dgl}_{\mathcal{C}}(A) = \sup \{ \text{dh}_A(M) : M \in A\text{-}\textbf{mod}(\mathcal{C})\} - \textit{the left global dimension of } A.
	$$
	\[
	\text{dgr}_{\mathcal{C}}(A) = \sup \{ \text{dh}_{A^{\text{op}}}(M) : M \in \textbf{mod}\text{-}A(\mathcal{C})\} - \textit{the right global dimension of } A.
	\]
	$$
	\text{db}(A) = \text{dh}_{A \hat{\otimes} A^{\text{op}}}(A) - \textit{the bidimension of }A.
	$$
	For a Fr\'echet algebra $A$ we can consider \textit{weak dimensions}.
	\[
	\begin{aligned}
	\text{w.dg}(A) &= \sup \{\text{w.dh}_A(M) : M \in A\text{-}\textbf{mod}(\textbf{Fr}) \} = \\ &= \sup \{\text{w.dh}_A(M) : M \in \textbf{mod}\text{-}A(\textbf{Fr}) \} - \textit{the weak global dimension of }A.
	\end{aligned}
	\]
	$$
	\text{w.db}(A) = \text{w.dh}_{A \hat{\otimes} A^{\text{op}}}(A) - \textit{the weak bidimension of }A.
	$$
\end{definition}
Unfortunately, it is not known to the author whether global dimensions depend on the choice of $\mathcal{C}$. We will denote
\[
\text{dgl}(A) := \text{dgl}_{\textbf{LCS}}(A), \quad \text{dgr}(A) := \text{dgr}_{\textbf{LCS}}(A).
\]
For more details the reader can consult \cite{book:1032964}.

The following theorem demonstrates one of the most important properties of homological dimensions.
\begin{theorem}{\cite[Proposition 3.5.5]{Helem1986}}
	\label{subadditivity}
	Let $A$ be a $\hat{\otimes}$-algebra. 
	If $0 \rightarrow X' \rightarrow X \rightarrow X'' \rightarrow 0$ is an admissible sequence of left $A\text{-}\hat{\otimes}$-modules, then
	\[
		\begin{aligned}
		\text{dh}_A(X) &\le \max\{\text{dh}_A(X'),\text{dh}_A(X'') \} \\
		\text{dh}_A(X') &\le \max\{\text{dh}_A(X),\text{dh}_A(X'')-1 \} \\
		\text{dh}_A(X'') &\le \max\{\text{dh}_A(X),\text{dh}_A(X')+1 \}.
		\end{aligned}
	\]
	In particular, $\text{dh}_A(X) = \max \{ \text{dh}_A(X'), \text{dh}_A(X'') \}$ except when $\text{dh}_A(X) < \text{dh}_A(X'') = \text{dh}_A(X') + 1$.
\end{theorem}

Moreover, the same estimates hold for weak homological dimensions of Fr\'echet modules over Fr\'echet algebras, for the proof see \cite{pirkovskiui2008weak}.
\begin{proposition}{\cite[Proposition 4.1, Corollary 4.4]{pirkovskiui2008weak}}
\label{weaksubadd}
Let $A$ be a Fr\'echet algebra, then for every $M \in A\text{-}\textbf{mod}(\textbf{Fr})$ we have
\[
\text{w.dh}_A(M) = \min\{ n \colon \text{Ext}^{n+1}_A(M, N^*) = 0 \text{ for every } N \in \textbf{mod}\text{-}A(\textbf{Fr}) \},
\] 
where $Y^*$ denotes the strong dual of $Y$.

As a corollary, for every admissible sequence
$0 \rightarrow X' \rightarrow X \rightarrow X'' \rightarrow 0$ of left Fr\'echet $A$-$\hat{\otimes}$-modules we have the following estimates:
	\begin{equation}
		\begin{aligned}
		\text{w.dh}_A(X) &\le \max\{\text{w.dh}_A(X'),\text{w.dh}_A(X'') \} \\
		\text{w.dh}_A(X') &\le \max\{\text{w.dh}_A(X),\text{w.dh}_A(X'')-1\} \\
		\text{w.dh}_A(X'') &\le \max\{\text{w.dh}_A(X),\text{w.dh}_A(X')+1\}.
		\end{aligned}
	\end{equation}
	
	In particular, $\text{w.dh}_A(X) = \max \{ \text{w.dh}_A(X'), \text{w.dh}_A(X'') \}$ except when $\text{w.dh}_A(X) < \text{w.dh}_A(X'') = \text{w.dh}_A(X') + 1$.
\end{proposition}

\begin{proposition}
	\label{free}
	Let $R \in \textbf{alg}(\mathcal{C})$, let $A$ be a $R$-$\hat{\otimes}$-algebra, which is free as a left $R$-module, and let $M$ be a right $A$-$\hat{\otimes}$-module.
	\begin{enumerate}[label=(\arabic*)]
		\item For every projective resolution of $M$ in $\textbf{mod}\text{-}R$
		$$
		0 \leftarrow M \leftarrow P_0 \leftarrow P_1 \leftarrow P_2 \leftarrow \dots
		$$
		the complex 
		\begin{equation}
		\label{tensresol}
		0 \leftarrow M \hat{\otimes}_R A \leftarrow P_0 \hat{\otimes}_R A \leftarrow P_1 \hat{\otimes}_R A \leftarrow P_2 \hat{\otimes}_R A \leftarrow \dots
		\end{equation}
		is a projective resolution of $M \hat{\otimes}_R A$ in the category of right $A$-$\hat{\otimes}$-modules. In particular, $$\text{dh}_{A^{\text{op}}}(M \hat{\otimes}_R A) \le \text{dh}_{R^\text{op}}(M).$$
		\item Moreover, if $\mathcal{C} \subset \textbf{Fr}$, then for every flat resolution of $M$ in $\textbf{mod}\text{-}R$
		$$
		0 \leftarrow M \leftarrow F_0 \leftarrow F_1 \leftarrow F_2 \leftarrow \dots
		$$
		the complex
		\begin{equation}
		\label{tensflatresol}
		0 \leftarrow M \hat{\otimes}_R A \leftarrow F_0 \hat{\otimes}_R A \leftarrow F_1 \hat{\otimes}_R A \leftarrow F_2 \hat{\otimes}_R A \leftarrow \dots
		\end{equation}
		is a flat resolution of $M \hat{\otimes}_R A$ in the category of right $A$-$\hat{\otimes}$-modules. In particular, $$\text{w.dh}_{A^{\text{op}}}(M \hat{\otimes}_R A) \le \text{w.dh}_{R^\text{op}}(M).$$
	\end{enumerate} 
\end{proposition}
\begin{proof}
	\indent
	\begin{enumerate}[label=(\arabic*)]
		\item  Lemma \ref{lemma1} implies that $P_i \hat{\otimes}_R A$ is a projective right $A$-$\hat{\otimes}$-module for all $i$. The complex \eqref{tensresol} is admissible, because the functor $(-) \hat{\otimes}_R A$ for a free $A$ preserves admissibility (due to the Proposition \ref{free admissible}), so it defines a projective resolution of $M \hat{\otimes}_R A$ in the category $\textbf{mod}\text{-}A$.
		\item Due to the Lemma \ref{lemma1} the modules $F_i \hat{\otimes}_R A$ are flat right $A$-$\hat{\otimes}$-modules for all $i$. Then the rest of the proof is the same as in (1).
	\end{enumerate}
\end{proof}

\begin{lemma}
	\label{trans}
	Let $R$ be a $\hat{\otimes}$-algebra, and consider an $R$-$\hat{\otimes}$-algebra $A$.
	\begin{enumerate}
		\item If $A$ is projective as a right $R$-$\hat{\otimes}$-module, and $M$ is projective as a right $A$-$\hat{\otimes}$-module, then $M$ is projective as a right $R$-$\hat{\otimes}$-module. In this case for every $X \in \textbf{mod}\text{-}A$ we have
		\begin{equation}
		\label{lower bound for proj dims}
			\text{dh}_{R^{\text{op}}}(X) \le \text{dh}_{A^{\text{op}}}(X).
		\end{equation}
		\item If $A$ is flat as a right $R$-$\hat{\otimes}$-module, and $M$ is flat as a right $A$-$\hat{\otimes}$-module, then $M$ is flat as a right $R$-$\hat{\otimes}$-module. If $R, A$ are Fr\'echet algebras, then for every $X \in \textbf{mod}\text{-}A(\textbf{Fr})$ we have
		\begin{equation}
		\label{lower bound for weak dims}
			\text{w.dh}_{A^{\text{op}}}(X) \le \text{w.dh}_{R^{\text{op}}}(X).
		\end{equation}
	\end{enumerate} 
\end{lemma}
\begin{proof}
	\begin{enumerate}
		\item[(1)]  Due to \cite[Theorem 3.1.27]{Helem1986} we can fix an isomorphism of right $R$-$\hat{\otimes}$-modules $A \oplus S \cong X \hat{\otimes} R$ for some right $R$-$\hat{\otimes}$-module $S$.
		
		Suppose that $M$ is a projective right $A$-$\hat{\otimes}$-module. Then due to \cite[Theorem 3.1.27]{Helem1986} there exists a right module $N$ such that
		$M \oplus N = E \hat{\otimes} A$ as right $A$-modules, but then
		\[
		M \oplus N \oplus (E \hat{\otimes} S) \cong E \hat{\otimes} A \oplus (E \hat{\otimes} S) \cong E \hat{\otimes} (X \hat{\otimes} R) = ( E \hat{\otimes} X) \hat{\otimes} R.
		\]
		\item[(2)] As in the proof of \cite[Lemma 7.2.2]{NonComm2001}, we notice that for any right $R$-$\hat{\otimes}$-module $N$ we have
		\[
		M \hat{\otimes}_R N \cong M \hat{\otimes}_A (A \hat{\otimes}_R N),
		\] 
		and this immediately implies our statement.
	\end{enumerate}
\end{proof}

%

\begin{lemma}
	\label{auto}
	Let $R$ is a $\hat{\otimes}$-algebra, $\alpha : R \rightarrow R$ is an
	automorphism and $M$ is a right $R$-module, then $\text{dh}_{R^{\text{op}}}(M_\alpha) =\text{dh}_{R^{\text{op}}}(M)$ and if $R$ is a Fr\'echet algebra, then  $\text{w.dh}_{R^{\text{op}}}(M_\alpha) =\text{w.dh}_{R^{\text{op}}}(M)$.
\end{lemma}
\begin{proof}
	The proof relies on the fact that $(\cdot)_\alpha : \textbf{mod}\text{-}R \rightarrow \textbf{mod}\text{-}R$ and ${}_\alpha (\cdot) : R\text{-}\textbf{mod} \rightarrow R\text{-}\textbf{mod} $ can be viewed as functors between $\textbf{mod}\text{-}R$ and $R\text{-}\textbf{mod}$, which preserve admissibility of morphisms, projectivity and flatness of modules. 
	
	Indeed, if $f : M \rightarrow N$ is an admissible module homomorphism, then $f_\alpha : M_\alpha \rightarrow N_\alpha$ is admissible, because $\square f_\alpha = \square f$, where $\square : \textbf{mod}\text{-}R \rightarrow \textbf{LCS}$ denotes the forgetful functor. The same goes for ${}_\alpha (\cdot)$.
	
	Let $P \in \textbf{mod}\text{-}R$ be projective. Then for any admissible epimorphism $\varphi : X \rightarrow Y$ we have the following chain of canonical isomorphisms:
	$$
	\text{Hom}(P_\alpha, Y) \simeq \text{Hom}(P, Y_{\alpha^{-1}}) \simeq \text{Hom}(P, X_{\alpha^{-1}}) \simeq \text{Hom}(P_\alpha, X).
	$$
	Let $F \in \textbf{mod}\text{-}R$ be flat. Then $F_\alpha \hat{\otimes}_R X \simeq F \hat{\otimes}_R ({}_{\alpha^{-1}} X)$ and we already know that ${}_{\alpha^{-1}} (\cdot)$ preserves admissibility.
\end{proof}

\section{Estimates for the bidimension and projective global dimensions of holomorphic Ore extensions}
\label{Estimates for the bidimension and projective global dimensions of holomorphic Ore extensions}
\subsection{Bimodules of relative differentials}
Firstly, let us give several necessary algebraic definitions:
\begin{definition}
	\label{def1}
	Let $S$ be an algebra, $A$ be an $S$-algebra and $M$ be an $A$-bimodule. Then an $S$-linear map $\delta : A \rightarrow M$ is called an $S$\textit{-derivation} if the following relation holds for every $a, b \in A$:
	$$
	\delta(ab) = \delta(a) b + a \delta(b).
	$$
\end{definition}

\begin{example}
	Let $\alpha$ be an endomorphism of an algebra $A$. Then an $\alpha$-derivation is precisely an $A$-derivation $\delta : A \rightarrow {}_{\alpha} A$.
\end{example}
\noindent
The following definition is due to J. Cuntz and D. Quillen, see \cite{CuntzQuil}:

\begin{definition}
	\label{reldifdef}
	Suppose that $S$ is an algebra and $(A, \eta)$ is an $S$-algebra, where $\eta : S \rightarrow A$ is an algebra homomorphism. Denote by $\overline{A} = A / \text{Im}(\eta(S))$ the $S$-bimodule quotient. Then we can define the bimodule of \textit{relative differential 1-forms} $\Omega^1_S(A) = A \otimes_S \overline{A}$. The elementary tensors in $\Omega^1_S(A)$ are usually denoted by $a_0 \otimes \overline{a_1} = a_0 d a_1$. The $A$-bimodule structure on $\Omega^1_S(A)$ is uniquely defined by the following relations:
	$$
	b \circ (a_0 d a_1) = b a_0 d a_1,\quad (a_0 d a_1) \circ b = a_0 d (a_1b) - a_0 a_1 d b.
	$$
\end{definition}
The bimodule of relative differential 1-forms together with the canonical $S$-derivation $$d_A : A \rightarrow \Omega^1_S(A),\quad d_A(a) = 1 \otimes \overline{a} = da$$ has the following universal property:
\begin{proposition}[\cite{CuntzQuil}, Proposition 2.4]
	\label{reldifuniv}
	For every $A$-bimodule $M$ and an $S$-derivation $D : A \rightarrow M$ there is a unique $A$-bimodule morphism $\varphi : \Omega^1_S(A) \rightarrow M$ such that the following diagram is commutative:
	\begin{equation}
		\begin{tikzcd}
		\Omega^1_R(A) \ar[r, dashrightarrow, "\exists! \, \varphi"] & M \\
		A \ar[u, "d_A"] \ar[ur, "D"']
		\end{tikzcd}
	\end{equation}
\end{proposition}

%
%

\begin{proposition}[\cite{CuntzQuil}, Proposition 2.5]
	\label{reldifex}
	The following sequence of $A$-bimodules is exact:
	\begin{equation}
		\begin{tikzcd}
			0 \ar[r] & \Omega^1_S A \ar[r, "j"] & A \otimes_S A \ar[r, "m"] & A \ar[r] & 0,
		\end{tikzcd}
	\end{equation}
	where 
	$j(a_0 \otimes \overline{a_1}) = j(a_0 d a_1) = a_0 a_1 \otimes 1 - a_0 \otimes a_1$ and 
	$m$ denotes the multiplication.
\end{proposition}

In the appendix of this paper we compute the bimodule of relative differential 1-forms of Ore extensions, see Proposition \ref{algebraic Ore extension}.

\subsection{A topological version of the bimodule of relative differentials}
	The following definition serves as a topological version of the Definition \ref{def1}.
\begin{definition}
	\label{topreldifdef}
	Let $R$ be a $\hat{\otimes}$-algebra, and suppose that $(A, \eta)$ is a $R$-$\hat{\otimes}$-algebra. Denote by \newline ${A / \overline{\text{Im}(\eta(R))} = \overline{A}}$ the $R$-$\hat{\otimes}$-bimodule quotient. Then we can define the \textit{(topological) bimodule of relative differential 1-forms} $\widehat{\Omega}^1_R(A) := A \otimes_R \overline{A}$. The elementary tensors are usually denoted by $a_0 \otimes \overline{a_1} = a_0 \text{d} a_1$.
	
	The structure of $A$-$\hat{\otimes}$-bimodule on $\widehat{\Omega}^1_R(A)$ is uniquely defined by the following relations:
	\[
	b \circ (a_0 \text{d}a_1) = ba_0 \text{d}a_1, (a_0 \text{d} a_1) \circ b = a_0 \text{d}(a_1b) - a_0 a_1 \text{d}b \quad \text{for every } a_0, a_1, b \in A.
	\]
\end{definition}
\textbf{Remark.} To avoid confusion with the algebraic bimodules of differential $1$-forms, here we use the notation $\widehat{\Omega}^1_R(A)$, unlike in \cite{PirkAM}.

We will need to recall the following propositions related to topological bimodule of relative differential $1$-forms.

\begin{theorem}{\cite[p. 99]{PirkAM}}
	For every $A$-$\hat{\otimes}$-bimodule $M$ and a continuous $R$-derivation $D : A \rightarrow M$ there exists a unique $A$-$\hat{\otimes}$-bimodule morphism $\varphi : \widehat{\Omega}^1_R(A) \rightarrow M$ such that the following diagram is commutative:
	\begin{equation}
		\begin{tikzcd}
			\widehat{\Omega}^1_R(A) \ar[r, dashrightarrow, "\exists! \, \varphi"] & M \\
			A \ar[u, "d_A"] \ar[ur, "D"']
		\end{tikzcd},
	\end{equation}
	where $d_A(a) = da$.
\end{theorem}

\begin{proposition}{\cite[Proposition 7.2]{PirkAM}}
	\label{splits}
	The short exact sequence
	\[
	\begin{tikzcd}
	0 \ar[r] & \widehat{\Omega}^1_R(A) \ar[r, "j"] & A \hat{\otimes}_R A \ar[r, "m"] & A \ar[r] & 0,
	\end{tikzcd}
	\]
	where $j(a_0 d a_1) = a_0 \otimes a_1  - a_0 a_1 \otimes 1$ and $m(a_0 \times a_1) = a_0 a_1$, splits in the categories $A$-\textbf{mod}-$R$ and $R$-\textbf{mod}-$A$.
\end{proposition}

\subsection{Holomorphic Ore extensions}

To give the definition of the holomorphic Ore extension of a $\hat{\otimes}$-algebra, associated with an endomorphism and derivation, we need to recall the definition of localizable morphisms.

\begin{definition}{\cite[Definition 4.1]{PirkAM}}
	Let $X$ be a LCS and consider a family $\mathcal{F} \subset \mathcal{L}(X)$ of continuous linear maps $X \rightarrow X$.
	
	Then a seminorm $\left\| \cdot \right\|$ on $X$ is called $\mathcal{F}$-\textit{stable} if for every $T \in \mathcal{F}$ there exists a constant $C_T > 0$, such that
	\[
	\left\| Tx \right\| \le C_T \left\| x \right\| \quad \text{for every } a \in A.
	\] 
\end{definition}

\begin{definition}
	\indent
	\begin{enumerate}[label=(\arabic*)]
		\item Let $X$ be a LCS.
		
		A family of continuous linear operators $\mathcal{F} \subset L(X)$ is called \textit{localizable}, if the topology on $X$ can be defined by a family of $\mathcal{F}$-stable seminorms.
		
		\item Let $A$ be an Arens-Michael algebra.
		
		A family of continuous linear operators $\mathcal{F} \subset L(A)$ is called $m$-\textit{localizable}, if the topology on $A$ can be defined by a family of submultiplicative $\mathcal{F}$-stable seminorms.
	\end{enumerate}
\end{definition}


Now we will state the theorem which proves the existence of certain $\hat{\otimes}$-algebras which would be reasonable to call the ``holomorphic Ore extensions''.

\begin{theorem}{\cite[Section 4.1]{PirkAM}}
	\label{holomoreex1}
	Let $A$ be a $\hat{\otimes}$-algebra and suppose that $\alpha : A \rightarrow A$ is a localizable endomorphism of $A$, $\delta : A \rightarrow A$ is a localizable $\alpha$-derivation of $A$.
	
	Then there exists a unique multiplication on the tensor product $A \hat{\otimes} \mathcal{O}(\mathbb{C})$, such that the following conditions are satisfied:
	\begin{enumerate}[label=(\arabic*)]
		\item The resulting algebra, which is denoted by $\mathcal{O}(\mathbb{C}, A; \alpha, \delta)$, is an $A$-$\hat{\otimes}$-algebra.
		\item The natural inclusion
		\[
		A[z; \alpha, \delta] \hookrightarrow \mathcal{O}(\mathbb{C}, A; \alpha, \delta)
		\]
		is an algebra homomorphism.
		\item For every Arens-Michael $A$-$\hat{\otimes}$-algebra $B$ the following natural isomorphism takes place:
		\[
		\text{Hom}(A[z; \alpha, \delta], B) \cong \text{Hom}_{A\text{-alg}}(\mathcal{O}(\mathbb{C}, A; \alpha, \delta), B).
		\]
	\end{enumerate}
	Moreover, let $\alpha$ be invertible, and suppose that the pair $(\alpha, \alpha^{-1})$ is localizable. Then there exists a unique multiplication on the tensor product $A \hat{\otimes} \mathcal{O}(\mathbb{C}^\times)$, such that the following conditions are satisfied:
	\begin{enumerate}[label=(\arabic*)]
		\item The resulting algebra, which is denoted by $\mathcal{O}(\mathbb{C}^\times, A; \alpha)$, is a $\hat{\otimes}$-algebra.
		\item The natural inclusion
		\[
		A[z; \alpha, \alpha^{-1}] \hookrightarrow \mathcal{O}(\mathbb{C}^\times, A; \alpha)
		\]
		is an algebra homomorphism.
		\item For every Arens-Michael $A$-$\hat{\otimes}$-algebra $B$ the following natural isomorphism takes place:
		\[
		\text{Hom}(A[z; \alpha, \alpha^{-1}], B) \cong \text{Hom}_{A\text{-alg}}(\mathcal{O}(\mathbb{C}^\times, A; \alpha), B).
		\]
	\end{enumerate}
	
	And if we replace the word ``localizable'' with ``$m$-localizable'' in this theorem, then the resulting algebras will become Arens-Michael algebras.
\end{theorem}

By considering $A_{\text{op}}[z; \alpha, \delta]$, we can formulate a version of the Theorem \ref{holomoreex1} for opposite Ore extensions, the proof is basically the same in this case. In fact, we need this to prove the following corollary:

%

\begin{corollary}
	\label{homolproperties}
	Let $R$ be a $\hat{\otimes}$-algebra with an endomorphism $\alpha : R \rightarrow R$ and a $\alpha$-derivation $\delta$ such that the pair $(\alpha, \delta)$ is localizable. 
	\begin{enumerate}[label=(\arabic*)]
		\item If $A = \mathcal{O}(\mathbb{C}, R; \alpha, \delta)$, then $A$ is free as a left $R$-$\hat{\otimes}$-module.
		\item If $\alpha$ is invertible, and $(\alpha, \alpha^{-1})$ is a localizable pair, then $A = \mathcal{O}(\mathbb{C}^\times, R; \alpha)$ is free as a left and right $R$-$\hat{\otimes}$-module.
		\item If $\alpha$ is invertible, and the pair $(\alpha, \delta)$ is $m$-localizable, then $A = \mathcal{O}(\mathbb{C}, R; \alpha, \delta)$ is free as a left and right $R$-$\hat{\otimes}$-module.
	\end{enumerate}
\end{corollary}
\begin{proof}
	\begin{enumerate}
		\item This already follows from the fact that $\mathcal{O}(\mathbb{C}, R; \alpha, \delta)$ is isomorphic to $A \hat{\otimes} \mathcal{O}(\mathbb{C})$.
		\item This is the immediate corollary of the \cite[Lemma 4.12]{PirkAM}: just notice that we can define
		\[
		\gamma : R \hat{\otimes} \mathcal{O}(\mathbb{C}) \rightarrow \mathcal{O}(\mathbb{C}) \hat{\otimes} R, \quad \gamma(r \otimes z^n) = z^n \otimes \alpha^{-n}(r),
		\]
		and it will be a continuous inverse of $\tau$. To finish the proof, we only need to notice that $\tau$ and $\gamma$ are isomorphisms of left and right $R$-$\hat{\otimes}$-modules, respectively.
		\item This follows from the above remark: let us denote the resulting ``opposite'' holomorphic Ore extensions by $\mathcal{O}_{\text{op}}(\mathbb{C}^\times, R; \alpha, \delta)$.
		
		Then the isomorpisms \ref{Ore isomorphisms} can be extended via the universal properties to the topological isomorphisms of $R$-$\hat{\otimes}$-algebras as follows: consider the algebra homomorphisms
		\[
		A[t; \alpha, \delta] \xrightarrow{\sim} A[t; \alpha^{-1}, -\delta \alpha^{-1}] \hookrightarrow \mathcal{O}_{\text{op}}(\mathbb{C}^\times, R; \alpha^{-1}, -\delta \alpha^{-1}),
		\]
		\[
		A[t; \alpha^{-1}, -\delta \alpha^{-1}] \xrightarrow{\sim} A[t; \alpha, \delta] \hookrightarrow \mathcal{O}(\mathbb{C}, R; \alpha, \delta).
		\]
		Now notice that the extensions are continuous and inverse on dense subsets of the holomorphic Ore extensions, therefore, they are actually inverse to each other.	However, the opposite Ore extensions are free as right $R$-$\hat{\otimes}$-modules by definition.
		
		Notice that to apply the universal properties we need the the morphisms to be \textit{m}-localizable, so that the holomorphic Ore extensions are Arens-Michael algebras.
	\end{enumerate}
\end{proof}

In this paper we compute the topological bimodules of relative differential 1-forms of holomorphic Ore extensions and smooth crossed products by $\mathbb{Z}$, see Propositions \ref{holomorphic Ore extension} and \ref{smooth extensions}.

\subsection{Upper estimates for the bidimension}
\begin{theorem}
	\label{upper estimates for bidimension holomorphic}
	Suppose that $R$ is a $\hat{\otimes}$-algebra, and $A$ is one of the two $\hat{\otimes}$-algebras:
	\begin{enumerate}[label=(\arabic*)]
		\item $A = \mathcal{O}(\mathbb{C}, R; \alpha, \delta)$, where the pair $\{\alpha, \delta\}$  is localizable.
		\item $A = \mathcal{O}(\mathbb{C}^\times, R; \alpha)$, where the pair $\{\alpha, \alpha^{-1}\}$ is localizable.
	\end{enumerate}
	Then we have 
	\[
	\text{db}(A^{\text{op}}) \le \text{db}(R^{\text{op}}) + 1.
	\]
\end{theorem}
\begin{proof}
Due to the Proposition \ref{splits} and Proposition \ref{holomorphic Ore extension}, we have the following sequence of $A$-$\hat{\otimes}$-bimodules, which splits in the categories $R\text{-}\textbf{mod}\text{-}A$ and $A\text{-}\textbf{mod}\text{-}R$:
\begin{equation}
	\label{exact1}
	\begin{tikzcd}
		0 \ar[r] & A_\alpha \hat{\otimes}_R A \ar[r, "j"] & A \hat{\otimes}_R A \ar[r, "m"] & A \ar[r] & 0,
	\end{tikzcd}
\end{equation}
where 
$m$ is the multiplication operator. Let
\begin{equation}
\label{something}
0 \leftarrow R \leftarrow P_0 \leftarrow \dots \leftarrow P_n \leftarrow 0
\end{equation}
be a projective resolution of $R$ in $R\text{-}\textbf{mod}\text{-}R$. Notice that \eqref{something} splits in $R\text{-}\textbf{mod}$ and $\textbf{mod}\text{-}R$, because all objects in the resolution are projective as left and right $R$-$\hat{\otimes}$-modules (\cite[Corollary 3.1.18]{Helem1986}). Therefore, we can apply the functors $A_\alpha \hat{\otimes}_R (-)$ and $A \hat{\otimes}_R (-)$ to \eqref{something} and the resulting complexes of $A$-$R$-$\hat{\otimes}$-bimodules are still admissible:
\begin{equation}
\label{projres}
0 \leftarrow A \leftarrow  A \hat{\otimes}_R P_0  \leftarrow \dots \leftarrow A \hat{\otimes}_R P_n \leftarrow 0
\end{equation}
\begin{equation}
\label{projres1}
0 \leftarrow A_\alpha \leftarrow  A_\alpha \hat{\otimes}_R  P_0 \leftarrow \dots \leftarrow A_\alpha \hat{\otimes}_R  P_n \leftarrow 0.
\end{equation}
Recall that $A$ is a free left $R$-$\hat{\otimes}$-module due to the Corollary \ref{homolproperties}, so the functor $(-) \hat{\otimes}_R A$ preserves admissibility, due to the Proposition \ref{free admissible}, therefore the following complexes of $A$-$\hat{\otimes}$-bimodules are admissible:
\begin{equation}
\label{projres2}
0 \leftarrow A \hat{\otimes}_R A  \leftarrow A \hat{\otimes}_R P_0 \hat{\otimes}_R A \leftarrow \dots \leftarrow A \hat{\otimes}_R P_n \hat{\otimes}_R A \leftarrow 0
\end{equation}
\begin{equation}
\label{projres3}
0 \leftarrow A_\alpha \hat{\otimes}_R A  \leftarrow A_\alpha \hat{\otimes}_R P_0 \hat{\otimes}_R A \leftarrow \dots \leftarrow A_\alpha \hat{\otimes}_R P_n \hat{\otimes}_R A \leftarrow 0
\end{equation}
Lemma \ref{lemma1} implies that \eqref{projres2} and \eqref{projres3} define projective resolutions for $A \hat{\otimes}_R A$ and $A_\alpha \hat{\otimes}_R A$. Now we can apply Theorem \ref{subadditivity} to \eqref{exact1}, so we get
$$
\text{db(A)} = \text{dh}_{A^\text{e}}(A) \le \max\{ \text{dh}_{A^\text{e}}(A \hat{\otimes}_R A), \text{dh}_{A^\text{e}}(A_\alpha \hat{\otimes}_R A)+1 \} \le n+1.
$$

In other words, we have obtained the desired estimate $$\text{db}(A) \le \text{db}(R) + 1.$$
\end{proof}


\subsection{Upper estimates for the right global and weak global dimensions}
Now we are prepared to state the theorem:
\begin{theorem}
	\label{upper estimates for global dimension holomorphic}
	
	Let $R$ be a $\hat{\otimes}$-algebra. Suppose that $A$ is one of the two $\hat{\otimes}$-algebras:
	\begin{enumerate}[label=(\arabic*)]
		\item $A = \mathcal{O}(\mathbb{C}, R; \alpha, \delta)$, where $\alpha$ is invertible, and the pair $\{\alpha, \delta\}$ is localizable.
		\item $A = \mathcal{O}(\mathbb{C}^\times, R; \alpha)$, where the pair $\{\alpha, \alpha^{-1}\}$ is localizable.
	\end{enumerate}
	Then the right global dimension of $A$ can be estimated as follows:
	\[
	\text{dgr}(A) \le \text{dgr}(R) + 1,
	\]
	and a similar estimate holds for the weak dimensions if $R$ is a Fr\'echet algebra:
	\[
	\text{w.dg}(A) \le \text{w.dg}(R) + 1.
	\]
\end{theorem}
\begin{proof}
	Suppose that $M$ is a right $A$-$\hat{\otimes}$-module. Then we can apply the functor $M\hat{\otimes}_A (-)$ to the sequence \eqref{exact1}. Notice that the resulting sequence of right $A$-$\hat{\otimes}$-modules
	\begin{equation*}
	\begin{tikzcd}
	0 \ar[r] & M \hat{\otimes}_A A_\alpha \hat{\otimes}_R A \ar[r, "\text{Id}_M \otimes j"] & M \hat{\otimes}_A A \hat{\otimes}_R A \ar[r, "\text{Id}_M \otimes m"] & M \hat{\otimes}_A A \ar[r] & 0
	\end{tikzcd}
	\end{equation*}
	is isomorphic to the sequence
	\begin{equation}
	\label{exact2}
	\begin{tikzcd}
	0 \ar[r] & M_\alpha \hat{\otimes}_R A \ar[r, "j'"] & M \hat{\otimes}_R A \ar[r, "m"] & M \ar[r] & 0.
	\end{tikzcd}
	\end{equation}
	
	Since \eqref{exact1} splits in $A\text{-}\textbf{mod}\text{-}R$, \eqref{exact2} splits in $\textbf{mod}\text{-}R$, in particular, this is an admissible short exact sequence. 
	
	Now notice that we can apply Theorem \ref{subadditivity} to \eqref{exact2}, so we get
	$$
	\text{dh}_{A^\text{op}}(M) \le \max\{ \text{dh}_{A^\text{op}}(M \hat{\otimes}_R A), \text{dh}_{A^\text{op}}(M_\alpha \hat{\otimes}_R A) + 1 \} \le \text{dh}_{R^\text{op}}(M) + 1
	$$
	due to the Corollary \ref{homolproperties}, Proposition \ref{free} and Lemma \ref{auto}. Hence, the following estimate holds: 
	$$
	\text{dgr}(A) \le \text{dgr}(R) + 1.
	$$
	
	For the weak dimensions we apply the second part of the Proposition \ref{weaksubadd} to the sequence \eqref{exact2}:
	\[
	\text{w.dh}_{A^\text{op}}(M) \le \max\{ \text{w.dh}_{A^\text{op}}(M \hat{\otimes}_R A), \text{w.dh}_{A^\text{op}}(M_\alpha \hat{\otimes}_R A)\} + 1 \le \text{w.dh}_{R^\text{op}}(M) + 1.
	\]
	
\end{proof}

%

%

\subsection{Lower estimates}
In order to obtain lower estimates, we need to formulate the following lemma:

\begin{proposition}
	\label{adm}
	Suppose that $R$ is a $\hat{\otimes}$-algebra, and $A$ is a $R$-$\hat{\otimes}$-algebra which is free as a left $R$-$\hat{\otimes}$-module. Also assume that there exists an isomorphism of left $R$-$\hat{\otimes}$-modules $\varphi : A \rightarrow R \hat{\otimes} E$ such that $\varphi(1) = 1 \otimes x$ for some $x \in E$. 
	
	Then $i : M \rightarrow M \hat{\otimes}_R A, i(m) = m \otimes 1$ is an admissible monomorphism for every $M \in \textbf{mod}\text{-}R$. In particular, it is a coretraction between the underlying locally convex spaces.
\end{proposition}
\begin{proof}
	Look at the following diagram:
	$$
	\begin{tikzcd}
		M \ar[r, "i"] \ar[rrr, bend right=20, "m \rightarrow m \otimes x"] & M \hat{\otimes}_R A \ar[r, "\text{Id}_M \otimes \varphi"] & M \hat{\otimes}_R R \hat{\otimes} E \ar[r, "\pi \otimes \text{Id}_E"] & M \hat{\otimes} E,
	\end{tikzcd}
	$$
	where $\pi : M \hat{\otimes}_R R \rightarrow M, \pi(m \otimes r) = mr$.
	 
	Due to the Hahn-Banach theorem there exists a functional $f \in E^*$ such that $f(x) = 1$, so the map $m \rightarrow m \otimes x$ admits a right inverse, which is uniquely defined by $n \otimes y \rightarrow f(y)n$, therefore $i$ as a mapping of lcs admits a right inverse too, because $\text{Id}_M \otimes \varphi$ and $\pi \otimes \text{Id}_E$ are invertible.
\end{proof}

\begin{proposition}
	\label{conditions for lower estimates}
	 Let $R$ be a Fr\'echet algebra, and assume that $\text{dgr}(R) < \infty$. Suppose that the following conditions hold:
	\begin{enumerate}[label=(\arabic*)]
	\item Let $A$ be a $R$-$\hat{\otimes}$-algebra which is a free left $R$-$\hat{\otimes}$-module, moreover, we can choose an isomorphism of left $R$-$\hat{\otimes}$-modules  $\varphi : A \rightarrow R \hat{\otimes} E$ in such a way that ${\varphi(1) = 1 \otimes x}$ for some $x \in E$.
	\item  $A$ is projective as a right $R$-$\hat{\otimes}$-module.
	\end{enumerate}
	Then $\text{dgr}_{\textbf{Fr}}(R) \le \text{dgr}_{\textbf{Fr}}(A)$ and $\text{w.dg}(R) \le \text{w.dg}(A)$.
\end{proposition}

\begin{proof}
	Fix a Fr\'echet module $M \in \textbf{mod}\text{-}R(\textbf{Fr})$ such that $\text{dh}_{R^{\text{op}}}(M) = \text{dgr}(R) = n$. The Proposition \ref{adm} states that the map $M \rightarrow M \hat{\otimes}_R A, \, m \rightarrow m \otimes 1$ is an admissible monomorphism, so there exists a short exact sequence
	\[
	0 \rightarrow M \xhookrightarrow{i} M \hat{\otimes}_R A \rightarrow N \rightarrow 0,
	\]
	where $N \cong (M \hat{\otimes}_R A) / i(M)$, which is precisely the cokernel of $i$. This sequence is admissible because the projection map is open and $i$ is an admissible monomorphism, we can apply the \cite[Proposition 3.1.8(III)]{Helem1986}. 
	
	\textbf{Remark.} If $M$ is an arbitrary $R$-$\hat{\otimes}$-module, then the map
	\[
	N \rightarrow (M \hat{\otimes}_R A) / i(M) \rightarrow ((M \hat{\otimes}_R A) / i(M))^\sim
	\]
	is \textit{not} open in general.
	
	Notice that $\text{dh}_{R^{\text{op}}}(M) = \text{dgr}(R)$ and $\text{dh}_{R^{\text{op}}}(N) \le \text{dgr}(R)$, therefore $\text{dh}_{R^{\text{op}}}(M \hat{\otimes}_R A) = \text{dgr}(R)$ due to the Proposition \ref{subadditivity}. 

	Now recall that $A$ is projective as a right $R$-$\hat{\otimes}$-module, therefore, due to the Lemma \ref{trans}, and the Proposition \ref{free} we have
	\[
	\text{dgr}(R) = \text{dh}_{R^{\text{op}}}(M \hat{\otimes}_R A) \stackrel{L\ref{trans}}{\le} \text{dh}_{A^{\text{op}}}(M \hat{\otimes}_R A) \stackrel{P\ref{free}}{\le} \text{dh}_{R^{\text{op}}}(M) = \text{dgr}(R).
	\]
	This immediately implies $\text{dgr}_{\textbf{Fr}}(R) \le \text{dgr}_{\textbf{Fr}}(A)$.
	
	By replacing $X$ with $X^*$ in the above argument and applying the Proposition \ref{weaksubadd} we get that $\text{w.dg}(R) \le \text{w.dg}(A)$.
\end{proof}

As a quick corollary from the Proposition \ref{conditions for lower estimates} and the Corollary \ref{homolproperties} we obtain lower estimates for the homological dimensions.

\begin{theorem}
	\label{lower estimates for global dimensions holomorphic}
	Let $R$ be a Fr\'echet algebra, and suppose that $\text{dgr}(R^{\text{op}}) < \infty$ and $A$ is one of the two $\hat{\otimes}$-algebras:
	\begin{enumerate}[label=(\arabic*)]
		\item $A = \mathcal{O}(\mathbb{C}, R; \alpha, \delta)$, where $\alpha$ is invertible, and the pair $(\alpha, \delta)$ is \textit{m}-localizable.
		\item $A = \mathcal{O}(\mathbb{C}^\times, R; \alpha)$, where the pair $(\alpha, \alpha^{-1})$ is localizable.
	\end{enumerate}
	Then the conditions of the Proposition \ref{conditions for lower estimates} are satisfied. As a corollary, we have the following estimates:
	$$
	\text{dgr}_{\textbf{Fr}}(R) \le \text{dgr}_{\textbf{Fr}}(A), \quad \text{w.dg}(R) \le \text{w.dg}(A).
	$$
\end{theorem}
\begin{proof}
	We have nothing to prove, because the Corollary \ref{homolproperties} ensures that $A$ is free as a left and right $R$-$\hat{\otimes}$-module, both conditions follow from this and the construction of holomorphic Ore extensions.
\end{proof}

%

\section{Homological dimensions of smooth crossed products by $\mathbb{Z}$}
\label{Homological dimensions of smooth crossed products}
First of all, let us recall the definition of the space of rapidly decreasing sequences:
\begin{definition}
	\[
	\begin{split}
	s &= \left\lbrace (a_n) \in \mathbb{C}^\mathbb{Z} : \left\| a \right\|_k =  \sup_{n \in \mathbb{Z}} |a_n| (|n|+1)^k < \infty ~ \forall k \in \mathbb{N} \right\rbrace = \\
	&\cong \left\lbrace (a_n) \in \mathbb{C}^\mathbb{Z} :  \left\| a \right\|^2_k = \sum_{n \in \mathbb{Z}} |a_n|^2 (|n|+1)^{2k} < \infty ~ \forall k \in \mathbb{N} \right\rbrace = \\
	&\cong \left\lbrace (a_n) \in \mathbb{C}^\mathbb{Z} : \left\| a \right\|_k = \sum_{n \in \mathbb{Z}} |a_n| (|n|+1)^k < \infty ~ \forall k \in \mathbb{N} \right\rbrace.
	\end{split}
	\]
\end{definition}

This is the Example 29.4 in \cite{vogt1997}. Clearly, this is a Fr\'echet space.

The following definitions and theorems are due to L. Schweitzer, see \cite{schweitzer1993dense} or \cite{phillips1994representable} for more detail.

\begin{definition}
	Suppose that $R$ is a Fr\'echet algebra. Let $G$ be one of the groups $\mathbb{R}, \mathbb{T}$ or $\mathbb{Z}$ and suppose that $\alpha : G \rightarrow \text{Aut}(R)$ is an action  of $G$ on $R$. Then $\alpha$ is called an $m$\textit{-tempered} action if the topology on $R$ can be defined by a sequence of submultiplicative seminorms $\{ \left\| \cdot \right\|_m : m \in \mathbb{N} \}$ such that for every $m$ there exists a polynomial $p$ satisfying
	\[
	\left\| \alpha_x(r) \right\|_\lambda \le |p(x)| \left\| r \right\|_\lambda
	\]
	for any $x \in G$ and $r \in R$.
\end{definition}

\begin{definition}
	\indent
	\begin{enumerate}
		\item[(1)] Let $E$ be a Hausdorff topological vector space. For a function $f : \mathbb{R} \longrightarrow E$ and $x \in \mathbb{R}$ we denote
		\[
		f'(x) := \lim_{h \rightarrow 0} \frac{f(x+h) - f(x)}{h}.
		\]
		(vector-valued differentiation).
		\item[(2)] Let $A$ be a Fr\'echet algebra with a fixed generating system of seminorms $\{ \left\| \cdot \right\|_\lambda, \lambda \in \Lambda\}$ Then we can define the following locally convex spaces:
		\[
		\mathscr{S}(\mathbb{Z}, A) = \left\lbrace f = (f^{(k)})_{k \in \mathbb{Z}} \in A^{\mathbb{Z}}  : \left\| f \right\|_{\lambda, k} := \sum_{n \in \mathbb{Z}} \left\| f^{(n)} \right\|_\lambda (|n|+1)^k < \infty \text{ for all } \lambda \in \Lambda, k \in \mathbb{N} \right\rbrace,
		\]
		\[
		\mathscr{S}(\mathbb{T}, A) = \left\lbrace f : \mathbb{T} \rightarrow A : \left\| f \right\|_{\lambda, k} := \sup_{z \in \mathbb{T}} \left\| f^{(k)}(z) \right\|_\lambda < \infty  \text{ for all } \lambda \in \Lambda, k \ge 0 \right\rbrace, 
		\]
		\[
		\mathscr{S}(\mathbb{R}, A) = \left\lbrace f : \mathbb{R} \rightarrow A : \left\| f \right\|_{\lambda, k, l} := \sup_{x \in \mathbb{R}} \left\| x^l f^{(k)}(x) \right\|_\lambda < \infty \text{ for all } \lambda \in \Lambda, k, l \ge 0 \right\rbrace. 
		\]
	\end{enumerate}
\end{definition} 


\begin{theorem}[\cite{schweitzer1993dense}, Theorem 3.1.7]
	\label{def of smooth ore extensions}
	Let $R$ be a Fr\'echet-Arens-Michael algebra with an \newline $m$-tempered action of one of the groups $\mathbb{R}, \mathbb{T}$ or $\mathbb{Z}$. Then the space $\mathscr{S}(G, R)$ endowed with the following multiplication:
	\[
	(f * g)(x) = \int_G f(y) \alpha_y(g(xy^{-1})) dy
	\]
	becomes a Fr\'echet-Arens-Michael algebra. This algebra is denoted by $\mathscr{S}(G, R; \alpha)$.
\end{theorem}
\noindent
\textbf{Remark.} The algebras $\mathscr{S}(G, R; \alpha)$ are, in general, non-unital for $G = \mathbb{R}, \mathbb{T}$. They are always unital for $G = \mathbb{Z}$.

\begin{proposition}
	\label{smooth extens is left and right free}
	Consider the following multiplication on $\mathscr{S}(G, R)$:
	\[
	(f *' g)(x) = \int_G \alpha_{y^{-1}}(f(xy^{-1})) g(y) \text{d}y.
	\]
	Then the following locally convex algebra isomorphism takes place:
	\[
	i : \mathscr{S}(G, R; \alpha) \rightarrow (\mathscr{S}(G, R), *'), \quad i(f)(x) = \alpha_{x^{-1}} (f(x)).
	\]
	In particular, when $G = \mathbb{Z}$, this is an isomorphism of unital algebras.
\end{proposition}

\begin{proof}
	The mapping $i$ is, obviously, a topological isomorphism of locally convex spaces. Now notice that
	\[
	\begin{split}
		(i(f) *' i(g))(x) &= \int_G \alpha_{y^{-1}} (i(f)(xy^{-1})) i(g)(y) \text{d}y = \int_G \alpha_{x^{-1}} (f(xy^{-1})) \alpha_{y^{-1}}(g(y)) \text{d}y = \\
		&= \int_G \alpha_{x^{-1}} (f(y^{-1})) \alpha_{y^{-1} x^{-1}}(g(yx)) \text{d} y = \alpha_{x^{-1}} \left( \int_G f(y^{-1}) \alpha_{y^{-1}}(g(yx)) \text{d}y \right) = \\
		&= i(f * g)(x),
	\end{split}
	\]
	therefore, $i$ is an algebra homomorphism.
\end{proof}

For example, let us consider a Fr\'echet-Arens-Michael algebra $R$ with a $m$-tempered action of $\mathbb{Z}$ and fix a generating family of submultiplicative seminorms $\{ \left\| \cdot \right\|_\lambda \ | \ \lambda \in \Lambda \}$ on $R$, such that for 
\[
\left\| \alpha_1^n(r) \right\|_\lambda \le p(n) \left\| r \right\|_\lambda, \quad (r \in R, n \in \mathbb{Z}),
\]
where $p$ is a polynomial. In this case the algebra $R$ is contained in $\mathscr{S}(\mathbb{Z}, R; \alpha)$:
\[
R \hookrightarrow \mathscr{S}(\mathbb{Z}, R; \alpha), \quad r \mapsto r e_0,
\] 
where $(re_i)_j := \delta_{ij} r$.

Hence, $\mathscr{S}(\mathbb{Z}, R; \alpha)$ becomes a unital $R$-$\hat{\otimes}$-algebra, and in the appendix we prove the Proposition \ref{smooth extensions}, which states that the structure of 
$\widehat{\Omega}^1_R( \mathscr{S}(\mathbb{Z}, R; \alpha) )$ is similar to the algebraic and holomorphic cases. This gives us an opportunity to formulate the following theorem:

\begin{theorem}
	Let $R$ be a Fr\'echet-Arens-Michael algebra with with a $m$-tempered $\mathbb{Z}$-action $\alpha$. If we denote $A = \mathscr{S}(\mathbb{Z}, R; \alpha)$, then we have
	\[
	\text{db}(A) \le \text{db}(R) + 1, \quad \text{dgr}(A) \le \text{dgr}(R) + 1, \quad \text{w.dg}(A) \le \text{w.dg}(R) + 1.
	\]
\end{theorem}
\noindent The proof of the theorem is very similar to the proofs of Theorems \ref{upper estimates for bidimension holomorphic} and \ref{upper estimates for global dimension holomorphic}.

As a simple corollary from the Proposition \ref{smooth extens is left and right free}, we get that $\mathscr{S}(\mathbb{Z}, R; \alpha)$ is free as a left and right $R$-$\hat{\otimes}$-module, and together with Proposition \ref{conditions for lower estimates} we obtain the lower estimates:

\begin{theorem}
	Let $R$ be Fr\'echet-Arens-Michael algebra with $\text{dgr}(R^\text{op}) < \infty$ and with a $m$-tempered $\mathbb{Z}$-action $\alpha$. Denote $A = \mathscr{S}(\mathbb{Z}, R; \alpha)$. Then the conditions of the Proposition \ref{conditions for lower estimates} are satisfied. In particular, we have
	\[
	\text{dgr}_{\textbf{Fr}}(R) \le \text{dgr}_{\textbf{Fr}}(A), \quad \text{w.dg}(R) \le \text{w.dg}(A).
	\]
\end{theorem}

\appendix

\section{Relative bimodules of differential 1-forms of Ore extensions}
\label{appendix A}
\begin{proposition}
	\label{algebraic Ore extension}
	Let $R$ be a $\mathbb{C}$-algebra. Suppose that
	\begin{enumerate}[label=(\arabic*)]
		\item $A = R[t; \alpha, \delta]$, where $\alpha : R \rightarrow R$ is an endomorphism and $\delta : R \rightarrow R$ is an $\alpha$-derivation.
		\item $A = R[t, t^{-1}; \alpha]$, where $\alpha : R \rightarrow R$ is an automorphism.
	\end{enumerate} 
	Then $\Omega^1_R(A)$ is canonically isomorphic as an $A$-bimodule to $A_\alpha \otimes_R A.$
\end{proposition}

\begin{proof} 
	The first part of the proof works for the both cases. Define the map $\varphi : A_\alpha \times A \rightarrow \Omega^1_R(A)$ as follows:
	\[
	\varphi(f, g) = f \text{d} (t g) - f t \text{d} g = f (\text{d} t) g,\quad (f, g \in A).
	\]
	This map is balanced, because
	\[
	\varphi(f, g) + \varphi(f, g') = \varphi(f, g + g'), \quad \varphi(f, g) + \varphi(f', g) = \varphi(f + f', g),\quad (f,f',g,g' \in A)
	\]
	and
	\[
	\varphi(f, r g) = f (\text{d} t) rg = f (\text{d}(tr)) g = f \text{d}(\alpha(r) t + \delta(r)) g = f \alpha(r) (\text{d} t) g = \varphi(f \circ r, g).
	\]
	Also we have
	\[
	h \varphi(f, g) = \varphi(hf, g), \quad \varphi(f, g) h = \varphi(f, gh).
	\]
	
	Therefore, $\varphi$ induces a well-defined homomorphism of $A$-bimodules $\varphi : A_\alpha \otimes_R A \rightarrow \Omega^1_R(A)$. 
	
	We will use the universal property of $\Omega^1_R(A)$ to construct the inverse morphism.
	\newline
	\begin{enumerate}[label=(\arabic*)]
		\item Suppose that $A = R[t; \alpha, \delta]$. Consider the following linear mapping:
		\[
		D : A \rightarrow A_\alpha \otimes_R A, \quad D(rt^n) = \sum_{k = 0}^{n-1} rt^k \otimes t^{n-k-1}.
		\]
		
		Now we want to show that $D$ is an $R$-derivation. First of all, notice that for any \newline ${f = \sum_{k=0}^m r_k t^k \in R[t; \alpha, \delta]}$ and $n \ge 0$ we have
		\begin{equation}
		\label{deriv1}
		D(f t^n) = \sum_{k = 0}^m r_k D(t^{k+n}) = \sum_{k = 0}^m r_k (D(t^k) t^n + t^k D(t^n)) = D(f) t^n + f D(t^n).
		\end{equation}
		It suffices to show that $D(t^n r) = D(t^n) r$, let us prove it by induction w.r.t. $n$:
		\begin{equation}
		\label{deriv2}
		\begin{aligned}
		D(t^{n+1}r) &= D(t^n \delta(r) + t^n \alpha(r) t) = D(t^n) \delta(r) + D(t^n) \alpha(r) t + t^n \alpha(r) \otimes 1 = \\ & = D(t^n) t r + t^n \alpha(r) \otimes 1 = D(t^n) t r + t^n \otimes r = (D(t^n) t + t^n D(t)) r = D(t^{n+1}) r.
		\end{aligned}
		\end{equation}

		\item Suppose that $A = R[t, t^{-1}; \alpha]$. Consider the following linear mapping:
		\[
		D : A \rightarrow A_\alpha \otimes_R A, \quad D(rt^n) = \begin{cases}
		\sum_{k = 0}^{n-1} rt^k \otimes t^{n-k-1}, & \text{ if } n \ge 0, \\
		-\sum_{k = 1}^{|n|} r t^{-k} \otimes t^{n+k-1}, & \text{ if } n < 0.
		\end{cases}
		\]
		As in the first case, this map turns out to be an $R$-derivation. Notice that it suffices to prove the following statements:
		\begin{itemize}
			\item $D(t^n r) = D(t^n) r \text{ for } n \in \mathbb{Z}, r \in R$.
			\item $D(t^m t^n) = D(t^m) t^n + t^m D(t^n) \text{ for } n, m \in \mathbb{Z}$.
		\end{itemize}
		If $n < 0$, then, to prove the first identity, we can use the argument, similar to \eqref{deriv2}, so let us concentrate on the proof of the second statement.
		If $m, n$ have the same sign, then we can repeat the argument in \eqref{deriv1}. Now suppose that $m > 0$ and $n < 0$. In this case we have
		\[
		D(t^m) t^n + t^m D(t^n) = \sum_{k = 0}^{m-1} t^k \otimes t^{n+m-k-1} - \sum_{k = 1}^{|n|} t^{m-k} \otimes t^{n+k-1}.
		\]
		If $m = |n|$, then everything cancels out and we get $0$. If $m < |n|$, then after the cancellations we get
		\[
		-\sum_{k = m+1}^{|n|} t^{m-k} \otimes t^{n+k-1} = - \sum_{j = 1}^{|m-n|} t^{-j} \otimes t^{n+m+j-1} = D(t^{|m-n|}).
		\]
		If $m > |n|$, then we get
		\[
		\sum_{k = 0}^{m+n-1} t^k \otimes t^{n+m-k-1} = D(t^{m+n}),
		\]
		and this finishes the argument.
	\end{enumerate}
	
	The rest of this proof works in the both cases. Notice that $\varphi \circ D = d_A$, because
	\[
	\varphi \circ D(rt^n) = r \varphi \left( \sum_{k = 0}^{n-1} t^k \otimes t^{n-k-1} \right) = r \sum_{k = 0}^{n-1} (t^{k} \text{d}(t^{n-k}) - t^{k+1} \text{d} (t^{n-k-1})) = r (\text{d}t^n) = d_A(r t^n)
	\] 
	for every $r \in A$, and for $n < 0$ we have
	\[
	\begin{aligned}
		\varphi \circ D(rt^n) &= -r \varphi \left( \sum_{k = 1}^{|n|} t^{-k} \otimes t^{n+k-1} \right) = -r \sum_{k = 1}^{|n|} (t^{-k} \text{d} (t^{n+k}) - t^{-k+1} \text{d}(t^{n+k-1})) = \\ &= r \text{d}(t^n) = d_A(r t^n).
	\end{aligned}
	\]
	Denote the extension of $D$ by $\widetilde{D}$, so $D = \widetilde{D} \circ d_A$. Therefore, we can derive from the universal property of $\Omega^1_R(A)$ that $\varphi \circ \widetilde{D} = \text{Id}_{\Omega^1_R(A)}$. And 
	\[
	\widetilde{D} \circ \varphi(a \otimes b) = a (\widetilde{D} \circ \varphi(1 \otimes 1)) b =  \widetilde{D}(dt) = a \otimes b \Rightarrow \widetilde{D} \circ \varphi = \text{Id}_{A_\alpha \otimes A}. 
	\]

	
\end{proof}

The following proposition was already proven by A. Yu. Pirkovskii, see \cite[Proposition 7.8]{PirkAM}, but we present another proof, which is similar to the proof of the Proposition \ref{algebraic Ore extension}; it even works in the case of \textit{localizable} morphisms. Moreover, the proof can be carried over to the case of smooth crossed products by $\mathbb{Z}$, as we will see later.

\begin{proposition}
	\label{holomorphic Ore extension}
	Let $R$ be an Arens-Michael algebra. Suppose that $A$ is one of the following $\hat{\otimes}$-algebras:
	\begin{enumerate}[label=(\arabic*)]
		\item $A = \mathcal{O}(\mathbb{C}, R; \alpha, \delta)$, where $\alpha: R \rightarrow R$ is an endomorphism and $\delta: R \rightarrow R$ is a $\alpha$-derivation, such that the pair $(\alpha, \delta)$ is a localizable pair of morphisms.
		\item $A = \mathcal{O}(\mathbb{C}^\times, R; \alpha)$, where $\alpha: R \rightarrow R$ is an automorphism, such that the pair $(\alpha, \alpha^{-1})$ is a localizable pair of morphisms.
	\end{enumerate} 
	Then $\widehat{\Omega}^1_R(A)$ is canonically isomorphic to $A_\alpha \hat{\otimes}_R A$.
\end{proposition}
\begin{proof}
	Fix a generating family of seminorms $\{ \left\| \cdot \right\|_\lambda : \lambda \in \Lambda \}$ on $R$ such that $\left\| \alpha(x)\right\|_\lambda  \le C \left\| x \right\|_\lambda $ and $\left\| \delta(x)\right\|_\lambda  \le C \left\| x \right\|_\lambda$.
	Define the map $A_\alpha \times A \rightarrow \widehat{\Omega}^1_R(A)$ as in the proof of the Proposition 0.1:
	\[
	\varphi(f, g) = f \text{d} (z g) - f z \text{d} g = f (\text{d} t) g,\quad (f, g \in A).
	\]
	This is a $R$-balanced map (the proof is literally the same as in the Proposition \ref{algebraic Ore extension}), also it is easily seen from the continuity of the multiplication on $A$ that this map is continuous.
	Therefore, this map induces a continuous $A$-$\hat{\otimes}$-bimodule homomorphism $A_\alpha \hat{\otimes}_R A \rightarrow \widehat{\Omega}^1_R(A)$.
	
	\begin{enumerate}[label=(\arabic*)]
		\item Suppose that $A = \mathcal{O}(\mathbb{C}, R; \alpha, \delta)$. Consider the following linear map:
		\[
		D : R[z; \alpha, \delta] \rightarrow A_\alpha \hat{\otimes}_R A, \quad D(rz^n) = \sum_{k = 0}^{n-1} rz^k \otimes z^{n-k-1}.
		\]
		For now it is defined on the dense subset of $A$; let us prove that this map is continuous. Fix $\lambda_1, \lambda_2 \in \Lambda$ and $\rho_1, \rho_2 \in \mathbb{R}_{\ge 0}$. Denote the projective tensor product of $\left\| \cdot \right\|_{\lambda_1, \rho_1}$ and $\left\| \cdot \right\|_{\lambda_2, \rho_2}$ by $\gamma$. Then for every $f = \sum_{k = 0}^m f_k z^k \in R[z; \alpha, \delta] \subset A$ we have
		\[
		\begin{aligned}
		\gamma(D(f)) & \le \sum_{k = 1}^m \gamma(D(f_k z^k)) = \sum_{k = 1}^m \gamma \left( \sum_{l = 0}^{k-1} f_k z^l \otimes z^{k-l-1} \right) \le \sum_{k = 1}^m \left\| f_k \right\|_{\lambda_1} \sum_{l = 0}^{k-1}  \rho_1^l \rho_2^{k-l-1} \le \\ & \le \sum_{k = 1}^m \left\| f_k \right\|_{\lambda_1} (2\max \{ \rho_1, \rho_2, 1 \})^k = \left\| f \right\|_{\lambda_1, 2\max \{ \rho_1, \rho_2, 1 \} }.
		\end{aligned}
		\]
		\item Suppose that $A = \mathcal{O}(\mathbb{C}^\times, R; \alpha)$. Consider the following linear map:			
		\[
		D : R[t, t^{-1}; \alpha] \rightarrow A_\alpha \hat{\otimes}_R A, \quad D(rz^n) = \begin{cases}
		\sum_{k = 0}^{n-1} r z^k \otimes z^{n-k-1}, & \text{ if } n \ge 0, \\
		-\sum_{k = 1}^{|n|} r z^{-k} \otimes z^{n+k-1}, & \text{ if } n < 0.
		\end{cases}
		\]
		For now it is defined on the dense subset of $A$; let us prove that this map is continuous. Fix $\lambda_1, \lambda_2 \in \Lambda$ and $\rho_1, \rho_2 \in \mathbb{R}_{\ge 0}$. Denote the projective tensor product of $\left\| \cdot \right\|_{\lambda_1, \rho_1}$ and $\left\| \cdot \right\|_{\lambda_2, \rho_2}$ by $\gamma$. Suppose that $n \ge 0$. Then we have
		\[
		\gamma( D(r z^n) ) = \gamma \left( \sum_{l = 0}^{n-1} r z^l \otimes z^{n-l-1} \right) \le
		\left\| r \right\|_{\lambda_1}  \sum_{l = 0}^{n-1} \rho_1^l \rho_2^{n-l-1} \le \left\| r z^n \right\|_{\lambda_1, 2 \max \{ \rho_1, \rho_2, 1 \}}.
		\]
		If $n < 0$, then
		\[
		\gamma( D(r z^n) ) = \gamma \left( \sum_{l=1}^{|n|} r z^{-l} \otimes z^{n+l-1} \right) \le 
		\left\| r \right\|_{\lambda_1} \sum_{l = 1}^{|n|} \rho_1^{-l} \rho_2^{n+l-1} \le \left\| r z^n \right\|_{\lambda_1, 2 \min \{ \rho_1, \rho_2, 1 \}}.
		\]
		Therefore, for every $f = \sum_{-m}^{m} f_k z^k \in R[z, z^{-1}; \alpha] \subset A$ we have
		\[
		\gamma (D(f)) =  \sum_{k=-m}^{m}  \gamma \left( D(f_k z^k) \right) \le \sum_{k=-m}^{m} \left\| f_k z^k \right\|_{\lambda_1, 2(\rho_1 + \rho_2 + 1)} = \left\| f \right\|_{\lambda_1, 2(\rho_1 + \rho_2 + 1)} . 
		\]
	\end{enumerate}

	Therefore, this map can be uniquely extended to the whole algebra $A$; we will denote the extension by $D$, as well. Notice $D$ is also an $R$-derivation: the equality $D(ab) - D(a)b - aD(b) = 0$ holds for ${R[z; \alpha, \delta] \times R[z; \alpha, \delta] \subset A \times A}$, which is a dense subset of $A \times A$. Therefore, $D(ab) = D(a)b + aD(b)$ for every $a, b \in A$.
	
	Notice that $\varphi \circ D = d_A$. Denote the extension of $D : A \rightarrow A_\alpha \hat{\otimes}_R A $ by $\widetilde{D} : \widehat{\Omega}^1_R(A) \rightarrow A_\alpha \hat{\otimes}_R A$, so $D = \widetilde{D} \circ d_A$. Therefore we can derive from the universal property of $\widehat{\Omega}^1_R(A)$ that $\varphi \circ \widetilde{D} = \text{Id}$. And $$\widetilde{D} \circ \varphi(a \otimes b) = a (\widetilde{D} \circ \varphi(1 \otimes 1)) b =  \widetilde{D}(dt) = a \otimes b.$$ Therefore, the equality $\widetilde{D} \circ \varphi = \text{Id}$ holds on a dense subset of $A_\alpha \hat{\otimes}_R A$, but $\widetilde{D} \circ \varphi$ is continuous, therefore, $\widetilde{D} \circ \varphi = \text{Id}$ holds everywhere on $A_\alpha \hat{\otimes}_R A$.

\end{proof}

\begin{proposition}
	\label{smooth extensions}
	Let $R$ be a Fr\'echet-Arens-Michael algebra and consider an $m$-tempered action $\alpha$ of $\mathbb{Z}$ on $R$. If we denote the algebra $\mathscr{S}(\mathbb{Z}, R; \alpha)$ by $A$, then $\widehat{\Omega}^1_R(A)$ is canonically isomorphic to $A_{\alpha_1} \hat{\otimes}_R A$.
\end{proposition}

\begin{proof}
	Fix a generating system of submultiplicative seminorms $\{ \left\| \cdot  \right\|_\lambda : \lambda \in \Lambda \}$ on $R$, such that
	\[
	\left\| \alpha_n(x) \right\|_\lambda = \left\| \alpha^n_1(x) \right\|_\lambda \le p(n) \left\| x \right\|_\lambda \quad (x \in R, \lambda \in \Lambda).
	\]
	
	Define the map $\varphi : A_{\alpha_1} \times A \rightarrow \widehat{\Omega}^1_R(A)$ as follows:
	\[
	\varphi(f, g) = f \text{d}(e_1 * g) - (f * e_1) \text{d}g = f (\text{d} e_1) g.
	\]
	It is a continuous $R$-balanced linear map (the proof is literally the same as in the previous propositions).

	Now consider the following linear map:
	\[
	D : A \rightarrow A_{\alpha_1} \hat{\otimes}_R A, \quad D(re_n) = \begin{cases}
	\sum\limits_{k = 0}^{n-1} re_k \otimes e_{n-k-1}, & \text{ if } n \ge 0 \\
	-\sum\limits_{k = 1}^{|n|} r e_{-k} \otimes e_{n+k-1}, & \text{ if } n < 0 
	\end{cases}.
	\]
	
	Let us prove that it is a well-defined and continuous map from $A$ to $A_{\alpha_1} \hat{\otimes}_R A$. Fix $\lambda_1, \lambda_2 \in \Lambda$ and $k_1, k_2 \in \mathbb{Z}_{\ge 0}$. Denote the projective tensor product of $\left\| \cdot \right\|_{\lambda_1, k_1}$ and $\left\| \cdot \right\|_{\lambda_2, k_2}$ by $\gamma$. Let $n \ge 1$, then we have
	
	\[
	\begin{aligned}
	\gamma(D(re_n)) &= \gamma \left( \sum\limits_{k = 0}^{n-1} re_k \otimes e_{n-k-1} \right) \le \left\| r \right\|_{\lambda_1}  (n^{k_2} + 2^{k_1} (n-1)^{k_2} + \dots n^{k_1}) \le \\ & \le \left\| r \right\|_{\lambda_1} (n^{\max \{ k_1, k_2 \}} + 2^{\max \{ k_1, k_2 \}} (n-1)^{\max \{ k_1, k_2 \}} + \dots + n^{\max \{ k_1, k_2 \}} ) \le \\ & \le \left\| r \right\|_{\lambda_1} n^{ 2{\max \{ k_1, k_2 \}} +1 } = \left\| r e_n \right\|_{\lambda_1,  2{\max \{ k_1, k_2 \}} +1 }.
	\end{aligned}
	\]
	For $n < 0$ the argument is pretty much the same:
	\[
	\begin{aligned}
	\gamma(D(re_n)) &= \gamma \left( \sum\limits_{k = 1}^{|n|} re_{-k} \otimes e_{n+k-1} \right) \le \left\| r \right\|_{\lambda_1}  (2^{k_1} (|n|+1)^{k_2} + 3^{k_1} |n|^{k_2} + \dots (|n|+1)^{k_1} 2^{k_2} ) \le 
	\\ & \le 
	\left\| r \right\|_{\lambda_1} (2^{\max \{ k_1, k_2 \}} (|n|+1)^{\max \{ k_1, k_2 \}} + \dots + (|n|+1)^{\max \{ k_1, k_2 \}} 2^{\max \{ k_1, k_2 \}} ) \le 
	\\ & \le 
	\left\| r \right\|_{\lambda_1} |n| (|n|+1)^{ 2 \max \{ k_1, k_2 \} } \le 
	2\left\| r \right\|_{\lambda_1} |n|^{4 \max \{ k_1, k_2 \} + 1} = 2 \left\| r e_n \right\|_{\lambda_1,  4{\max \{ k_1, k_2 \}} +1 }.
	\end{aligned}		
	\]
	It is easily seen that for every $f \in R \hat{\otimes} c_{00}$ we have
	\[
	\gamma(Df) \le \sum_{m \in \mathbb{Z}} \gamma( D (f^{(m)} e_m)) \le 
	2 \sum_{m \in \mathbb{Z}} \left\| f^{(m)} e_m \right\|_{\lambda_1, 4{\max \{ k_1, k_2 \}} +1} = 
	2 \left\| f \right\|_{\lambda_1, 4{\max \{ k_1, k_2 \}} +1}. 
	\]
	Then $D$ is a $R$-derivation which can be uniquely extended to the whole algebra $A$, the extension $\widetilde{D}$ is the inverse of $\varphi$, the proof is the same as in the Proposition \ref{holomorphic Ore extension}.
\end{proof}

\printbibliography
\Addresses

\end{document}